\documentclass[a4paper, 11pt, reqno]{amsart}
\usepackage[usenames,dvipsnames]{color}
\usepackage{amssymb, amsmath, latexsym, graphics, graphicx}

\makeatother

\newcommand\trop{\mathbb{T}}

\renewcommand{\S}{{\mathcal{S}}}

\usepackage[utf8]{inputenc}
\usepackage{amsmath}
\usepackage{amsfonts}
\usepackage{amsthm}
\usepackage{enumerate}

\newtheorem{thm}{Theorem}[section]
\newtheorem{cor}[thm]{Corollary}
\newtheorem{lem}[thm]{Lemma}
\newtheorem{prop}[thm]{Proposition}
\theoremstyle{definition}

\newcommand{\N}{\mathbb{N}}
\newcommand{\R}{\mathbb{R}}

\newcommand{\T}{\mathbb{T}}

\newcommand{\B}{\mathbb{B}}

\newcommand{\cS}{\mathcal{S}}

\newcommand{\cB}{\mathcal{B}}
\newcommand{\Nmax}{\N_\mathrm{max}}
\newcommand{\Nmin}{(-\N)_\mathrm{max}}
\newcommand{\kmax}{[k]_\mathrm{max}}
\newcommand{\kmin}{[-k]_\mathrm{max}}

\newcommand{\ceil}[1]{\left \lceil #1 \right \rceil}

\begin{document}
\title[Permutability of Matrices over Bipotent Semirings]{Permutability of Matrices over \\ Bipotent Semirings}
\date{\today}
\thanks{}
\maketitle
\begin{center}
THOMAS AIRD\footnote{Email \texttt{Thomas.Aird@manchester.ac.uk}.}
AND MARK KAMBITES\footnote{Email \texttt{Mark.Kambites@manchester.ac.uk}.} \\ \ \\
Department of Mathematics, University of Manchester, \\
Manchester M13 9PL, UK.

\end{center}

\begin{abstract}
We study permutability properties of matrix semigroups over commutative bipotent semirings (of which the best-known example is the \textit{tropical semiring}). We prove that every such semigroup is weakly permutable (a result previous stated in the literature, but with an erroneous proof) and then proceed to study in depth the question of when they are strongly permutable (which turns out to depend heavily on the semiring). Along the way we classify monogenic bipotent semirings and describe all isomorphisms between truncated tropical semirings.
\end{abstract}

Commutative bipotent semirings appear naturally in many areas of mathematics; for example, the \textit{boolean} semiring has important applications in computer science \cite{Golan}, while \textit{tropical} and related semirings  have found applications in areas as diverse as algebraic geometry, geometric group theory, automata and formal languages, and combinatorial optimization and control theory
 \cite{Bio,Con,Geom}. Many of the problems which arise naturally in these areas involve finite systems of linear (over the semiring) equations and can therefore be formulated in terms of matrix operations; understanding the structure of matrix algebra over these semirings is thus vital for applications, and much recent research has been devoted to this topic.  An additional motivation comes from abstract semigroup theory, where there is increasing evidence \cite{Zur0,Zur2,JKPlactic} that tropical and related semirings are natural carriers for representations of important classes of semigroups and monoids which, due to their structural properties, do not admit faithful finite dimensional representations over fields.

\par In this paper, we focus on two algebraic finiteness conditions for semigroups of matrices over bipotent semirings: weak permutability and permutability. A semigroup $S$ is called \textit{weakly permutable} if there exists a $k \geq 2$ such for any $s_1,\dots,s_k \in S$ there exist permutations $\sigma \neq \tau$ of $\{1,\dots,k\}$ such that $s_{\sigma(1)}s_{\sigma(2)}\cdots s_{\sigma(k)} = s_{\tau(1)}s_{\tau(2)}\cdots s_{\tau(k)}$.
A semigroup $S$ is called \textit{permutable} (or sometimes \textit{strongly permutable}) if there exists a $k \geq 2$ such for any $s_1,\dots,s_k \in S$ there exists a permutation $\sigma$ of $\{1,\dots,k\}$ such that $s_{\sigma(1)}s_{\sigma(2)}\cdots s_{\sigma(k)} = s_{1}s_{2}\cdots s_{k}$. 
We note here a few key facts about these properties; for a comprehensive introduction the reader is directed to \cite[Chapter 19]{Okninski}. Notice that every strongly permutable semigroup $S$ is weakly permutable by taking $\tau$ to be the identity permutation.  Every finite semigroup is clearly strongly permutable, as is every commutative semigroup. Indeed, weak and strong permutability may be thought of as very weak commutativity conditions. It is easy to see that if a semigroup $S$ is weakly [strongly] permutable then every subsemigroup of $S$ and every homomorphic image of $S$ is also weakly [strongly] permutable. Permutability conditions are of interest in general because of connections with polynomial identities in semigroup algebras \cite[Chapter~19]{Okninski}, and are lent additional importance in these particular semigroups by interest in representations over semirings: any permutability condition satisfied by matrix semigroups poses an obstruction to faithfully representing semigroups not satisfying the condition. 

\par We begin, in Section 1, by establishing some structural results about commutative bipotent semirings which will be useful in our subsequent analysis. These include a simple classification of the monogenic examples, which may be of independent interest.

\par In Section 2 we proceed to look at weak permutability, proving that every full matrix semigroup, and hence every matrix semigroup, over a commutative bipotent semiring is weakly permutable. This fact was first stated by d'Alessandro and Pasku \cite{Ales} but there is an error (described below) in their proof.  

\par In Section 3 we turn our attention to (strong) permutability. If the semiring has an element of infinite multiplicative order (or more generally, elements of unbounded multiplicative order) we prove (Theorem~\ref{unbounded}) that the full matrix semigroup and upper triangular matrix semigroups are not strongly permutable in any dimension greater than $1$. This applies in particular to the tropical and many related semirings.
On the other hand, semirings with bounded multiplicative order exhibit a range of behaviours, with apparently similar semirings sometimes differing
quite dramatically. Matrix semigroups over \textit{chain semirings}, which are multiplicatively as well as additively idempotent, are strongly permutable in all dimensions (Corollary~\ref{chains}). 
\par Section 4 is devoted to the class of \textit{truncated tropical semirings}, where it transpires that the full matrix semigroups can be strongly permutable in all dimensions (Theorem~\ref{12perm}),  only in dimension $1$ (Corollary~\ref{unboundedtruncated}) or, interestingly, only in dimensions $1$ and $2$ (Theorem~\ref{1zperm}). Similar results are obtained for the monoids of upper triangular and upper unitriangular matrices. In the course of our study we describe all isomorphisms between truncated tropical semirings.

\par Throughout this paper we write $\mathbb{N}$ for the set of natural numbers excluding $0$. For $n \in \mathbb{N}$ we write $[n]$ for the discrete
interval $\mathbb{N} \cap [1,n]$, and $\cS_n$ for the symmetric group on the set $[n]$.

\par \textbf{Acknowledgements.} The authors thank Marianne Johnson for some helpful conversations and comments on the
draft.

\section{Commutative Bipotent Semirings}
\par For our purposes a \emph{semiring} $S$ is a non-empty set with two binary operations --- addition and multiplication --- such that both operations are associative, addition is commutative, and multiplication distributes over addition on both sides.  A semiring is called \emph{commutative} if the multiplication is commutative and \emph{bipotent} if $x + y$ is always either $x$ or $y$. A bipotent semiring admits a natural linear order
defined by $x \leq y$ if and only if $x+y = y$, and the distributive laws mean exactly that multiplication respects this order, giving rise to a totally ordered semigroup. Conversely, every totally ordered semigroup gives rise to a bipotent semiring, by taking the semigroup operation as multiplication and defining the sum to be maximum with respect to the order. Bipotent semirings are thus, at one level, the same thing as totally ordered semigroups, but the two viewpoints lead naturally to rather different questions; in particular the semiring viewpoint leads to the study of linear algebra and matrices. Our main interest is in commutative bipotent semirings, although some of our results will extend to the non-commutative case.

Some authors insist that a semiring should have a zero (that is an element which is a multiplicative zero and an additive identity) and/or a (multiplicative) identity element, but most of our results will not require these. In fact it is easy to see that any commutative bipotent semiring $S$ without a zero element can have one ``adjoined'', that is, can be embedded in a commutative bipotent semiring with one extra element $0$ which is a zero. We write $S^0$ for this semiring; if $S$ already has a zero element then we define $S^0 = S$ and use $0$ to denote the zero element of $S$.
On the other hand, the corresponding statement is not true for identity elements:
\begin{prop}\label{noidentity}
There exists a commutative bipotent semiring $S$ without identity which cannot be embedded in any bipotent semiring with identity. 
\end{prop}
\begin{proof}
Let $S = \{ a , b , c \}$ be the commutative bipotent semiring such that $c \geq b \geq a$, all elements are multiplicatively idempotent, and all non-idempotent products are $b$.  It is straightforward to verify that the given operations respect the associative and distributive laws. Now suppose we can embed $S$ in a bipotent semiring with identity $1$, and consider where $1$ lies in the order. If $1 > b$, then $a(1+b) = a1 = a$, but by the distributive law $a(1+b) = a1 + ab = a+b = b$ giving a contradiction. On the other hand, if $1 < b$, then $c(1+b) = cb = b$, but similarly by the distributive law $c(1+b) = c1 + cb = c+b = c$, giving a contradiction. Thus, we cannot embed $S$ into a bipotent semiring with identity.
\end{proof}
Notwithstanding the impossibility in general of adjoining an identity element, it is sometimes convenient to introduce ``the identity''
as a \textit{purely notational} device. If $S$ is a commutative semiring without identity, we define $S^1$ to be $S \cup \lbrace 1 \rbrace$ where $1$ is a new symbol, and define $1x=x$  for all $x$ and $1+1 = 1$ and also (where $S$ has a $0$) $1+0 = 1$, but leave other sums involving $1$ undefined. We caution that this structure is not a semiring, since addition is only partially defined. Again, if $S$ already has an identity we set $S^1 = S$ and use $1$ to denote the existing identity. We write $S^{01}$ for $(S^0)^1$.

A \textit{subsemiring} is a subset closed under addition and multiplication; note that even if $S$ has zero and/or identity
elements, subsemirings are not required to contain them.
If $a \in S$ then we write $\langle a \rangle$ for the (\textit{monogenic}) subsemiring of $S$ generated by $a$ (that is, the intersection of all subsemirings
containing $a$). If $S$ is bipotent then $\langle a \rangle$ coincides with the
multiplicative subsemigroup of $S$ generated by $a$, in other words, the set of positive powers of $a$. The
(multiplicative) \textit{order} of $a$ is defined to be the cardinality of the set of positive powers of $a$, which when $S$ is bipotent is the cardinality of $\langle a \rangle$.

We will consider in particular the following examples of commutative bipotent semirings; some of these merit study due to external
applications, some arise naturally in the general theory, and others are included to illustrate the full range of possible behaviours:
\begin{itemize}
\item The \textit{tropical} (or \textit{max-plus}) \textit{semifield} $\T$ consists of the real numbers augmented with $-\infty$, with maximum as its addition and addition as
its multiplication; it has applications in numerous areas including; biology \cite{Bio}, control theory \cite{Con} and algebraic geometry \cite{Geom}.
The tropical semifield admits isomorphic manifestations as the \textit{min-plus semifield} (the real numbers augmented with $+\infty$ under \textit{minimum} and classical addition) and the \textit{max-times semifield} (the non-negative real numbers under maximum and classical \textit{multiplication}).
\item The \textit{tropical natural number semiring} $\Nmax$ is the subsemiring of $\T$ consisting of natural numbers; it has applications in areas such as formal language theory and automata theory \cite{Auto}.
\item The \textit{tropical negative natural number semiring} $\Nmin$ is the subsemiring of $\T$ consisting of the negative integers. (It is isomorphic
to the natural numbers under \textit{minimum} and classical addition.)
\item For $x, y \in \mathbb{R}$ with $0 \leq x < y$ the \textit{truncated tropical semiring} $\T_{[x,y]}$ consists of the real interval $[x,y]$ augmented with $0$ and $-\infty$ with operations maximum and \textit{$y$-truncated addition} given by $ab = \min(a+b, y)$ where $+$ here denotes classical addition.
\item For $k \in \mathbb{N}$ the \textit{truncated tropical natural number semiring} $\kmax$ consists of the set $[k] = \lbrace 1, \dots, k \rbrace$ with operations defined
as in $\T_{[1,k]}$.
\item For $k \in \mathbb{N}$ the \textit{truncated tropical negative natural number semiring} $\kmin$ consists of the set $\lbrace -k, \dots, -1 \rbrace$ with operations maximum and \textit{$(-k)$-truncated addition} given by $a \cdot b = \max(a+b,-k)$. (Note that $[-1]_{\max}$ and $[1]_{\max}$ are both trivial and therefore isomorphic to each other.)
\item Any linearly ordered set admits the structure of a commutative bipotent semiring, with maximum as addition and minimum as multiplication. We call these \emph{chain semirings}. A prominent example is the 2-element chain semiring, the \textit{boolean semifield}, which is isomorphic to the semiring with two elements $True$ and $False$ with operations ``and'' and ``or'', and has natural applications in logic and computer science \cite{Golan}.
\end{itemize}

For any semiring $S$ and $n \in \mathbb{N}$, the set of $n \times n$ matrices over $S$ forms a semiring under matrix multiplication and addition induced from $S$ in the usual way. Note that $M_n(S)$ will typically be neither commutative nor bipotent (even when $S$ is both). Our main interest here is in the multiplicative semigroup of $M_n(S)$. We also define $UT_n(S)$ to be the subsemiring of $M_n(S^0)$ consisting of matrices with $0$ below the main diagonal and elements from $S$ on and above the main diagonal.
We write $U_n(S)$ for the semiring of matrices over $S^{01}$ which have $0$ below the main diagonal, $1$ on the main diagonal and elements from $S$ above the main diagonal; note that even if the $1$ is adjoined the partial addition in $S^{01}$ is defined for sufficiently many values to enable matrix addition and multiplication on this set. Again, our principal interest is in the structure of $UT_n(S)$ and $U_n(S)$ as multiplicative semigroups. Note that $M_1(S) = UT_1(S)$ is isomorphic to the multiplicative semigroup of $S$, while $U_1(S)$ is the trivial monoid and $U_2(S)$ is isomorphic to the additive semigroup of $S$.

\begin{lem} \label{order}
Let $S$ be a bipotent semiring. If an element $x \in S$ has finite multiplicative order (that is, has finitely many distinct powers) then it has period $1$ (that is, $x^k = x^{k+1}$ for some $k \in \N$).
\end{lem}
\begin{proof}
Let $x \in S$ have finite multiplicative order. Then there exist $r, m \in \N$ such that $x^m = x^{m+r}$. If $r=1$ we are done, so assume $r>1$. As $S$ is bipotent we have that the sum $x^m + \dots + x^{m+r-1} = x^k$ for some $k$ between $m$ and $m+r-1$. But now by distributivity and commutativity of addition,
\begin{align*}
x^{k+1} \ &= \ x(x^m + \dots + x^{m+r-2} + x^{m+r-1}) \\
&= \ x^{m+1} + \dots + x^{m+r-1} + x^m \\
&= \ x^k.
\end{align*}
\end{proof}

The following lemma describes all the possible bipotent semirings generated by a single element:
\begin{lem} \label{Ninside}
Let $S$ be a bipotent semiring. If $a \in S$ and $\langle a \rangle$ is the monogenic subsemiring generated by $a$, then
\[\langle a \rangle \cong
\begin{cases}
\Nmax  &\text{ if } a \text{ has infinite order and } a < a^2; \\
\Nmin &\text{ if } a \text{ has infinite order and } a^2 < a; \\
\kmax &\text{ if } a \text{ has order } k \in \mathbb{N} \text{ and } a \leq a^2; \\
\kmin &\text{ if } a \text{ has order } k \in \mathbb{N} \text{ and } a^2 \leq a.
\end{cases}\]
\end{lem}
We remark that the four cases above are comprehensive but not quite mutually exclusive: in the case that $a$ has order $1$ we have $a = a^2$ and $\langle a \rangle$
is isomorphic to both $[1]_{\max}$ and $[-1]_{\max}$.
\begin{proof}
First suppose $a \leq a^2$. Define a map
$$\phi \ : \ \Nmax \to \langle a \rangle, \ \ n \mapsto a^n.$$
This map is surjective (because of our observation that, in a bipotent semiring, $\langle a \rangle$ coincides with the multiplicative
semigroup generated by $a$) and preserves multiplication because of basic properties of powers.
Now let $n, m \in \N$ and suppose without loss of generality that $n \geq m$. Since $a \leq a^2$ we have $a^k \leq a^{k+1}$ for all $k$ (because the total order is compatible with multiplication) and hence $a^m \leq a^n$ (because $m \leq n$ and the order is transitive). Therefore
$$
\phi(\max(n,m)) \ = \ \phi(n) \ = \ a^n \ = \ a^n + a^m \ = \ \phi(n) + \phi(m).
$$
If $a$ has infinite order then $\phi$ is injective, and we have shown that it is an isomorphism from $\Nmax$ to $\langle a \rangle$.

\par If $a$ has finite order $k$ then let $\varphi$ be the restriction of $\phi$ to the subset $[k]$.
Clearly $\varphi$ is a bijection. Since the semiring addition (in other words, the order) on $\kmax$ is the restriction of that on 
$\Nmax$, the fact that $\varphi$ preserves semiring addition follows from the fact that $\phi$ does. Now let $n, m \in \N$ and suppose without loss of generality that $n \geq m$. Then 
$$\varphi(nm) \ = \ a^{n + m} \ = \ a^n a^m \ = \ \varphi(n) \varphi(m)$$
for all $n, m \in \N$. The first equality here holds because if $n+ m \geq k$ then $a^{n+m} = a^k$, as $a$ has period 1 by Lemma~\ref{order}.
Hence, $\varphi$ is an isomorphism between $\kmax$ and $\langle a \rangle$.

Similarly if $a^2 \leq a$ then we define
$$\psi \ : \ \Nmin \to \langle a \rangle, \ \ n \mapsto a^{-n}.$$
Again $\psi$ is surjective. This time for negative integers $n \geq m$ we use $a^2 \leq a$ to deduce that $a^{-m} \leq a^{-n}$ so 
$$\psi(\max(n,m)) = \psi(n) = a^{-n} = a^{-n} + a^{-m} = \psi(n) + \psi(m).$$
and $\psi$ preserves semiring addition. If $a$ has infinite order then $\psi$ is injective and preserves the semiring multiplication, so it is
an isomorphism between $\Nmin$ and $\langle a \rangle$. If $a$ has finite order $k$ then an entirely similar argument to that above shows that the restriction of $\psi$ to $-[k]$ is an isomorphism between $\kmin$ and $\langle a \rangle$.
\end{proof}

\section{Weak Permutability}

\par In this section we briefly consider weak permutability, showing that any semigroup of matrices over a commutative bipotent semiring always has this
property. This result was first stated by d'Allesandro and Pasku \cite{Ales}, but Taylor \cite{TaylorThesis} identified an error in their proof. The error and its consequences are
discussed below. Our proof is, nonetheless, inspired by their method.

\begin{prop}
Let $S$ be a commutative bipotent semiring. Then $M_n(S)$ is weakly permutable for all $n \in \N$.
\end{prop}

\begin{proof}
Fix $n \in \N$. Let $\Gamma_n$ denote the complete directed graph (with loops)
on the set $[n]$. We identify edges in $\Gamma_n$ with pairs in $[n] \times [n]$
in the obvious way; in particular we will index the entries of $n \times n$
matrices by edges in $\Gamma_n$.

Let $\Pi$ denote the set of $n \times n$ matrices whose entries are edges 
from $\Gamma_n$ (that is, pairs from $[n] \times [n]$). Let
$c = |\Pi| = n^{2n^2}$. Choose $k$ large enough that $k! > c^k$.

Consider a finite sequence of $k$ matrices of size $n \times n$ 
over the semiring $S$, say $M_1, \dots, M_k$. For a permutation $\sigma$ in
the symmetric group $\cS_k$, write
$$M_\sigma \ = \ M_{\sigma(1)} M_{\sigma(2)} M_{\sigma(3)} \cdots M_{\sigma(k)}.$$
We must show that there are 
distinct permutations $\sigma, \tau \in \cS_k$ with $M_\sigma = M_\tau$.

We define a function $\pi : \cS_k \to \Pi^{[k]}$ (where $\Pi^{[k]}$
denotes the set of functions from $[k]$ to $\Pi$) as follows. For each
$\sigma \in \cS_k$ and each $x, y \in [n]$, consider the $(x,y)$ entry
of the matrix $M_\sigma$. It follows from the definition of matrix
multiplication and the fact $S$ is bipotent that there is at
least one path $p_1, \dots, p_k$ of length $k$ from $x$ to $y$ in $\Gamma_n$
such that this entry is given by
\begin{equation}\label{eqa}
{(M_\sigma)}_{x,y} \ = \ {(M_{\sigma(1)})}_{p_1} {(M_{\sigma(2)})}_{p_2} \cdots {(M_{\sigma(k)})}_{p_k}.
\end{equation}
Choose any such path, and for each $i \in [k]$ define the $(x,y)$ entry of
$\left( \pi (\sigma) \right) (i)$ to be the edge $p_{\sigma^{-1}(i)}$ (that is, the
edge indexing the entry of $M_i$ which contributes in the computation of the $(x,y)$ entry of
$M_\sigma$). Thus reordering the terms in \eqref{eqa} we have
\begin{equation*}\label{eqb}
(M_\sigma)_{x,y} = 
(M_1)_{\left( \pi(\sigma) \right)(1)_{x,y}} 
(M_2)_{\left( \pi(\sigma) \right)(2)_{x,y}} \cdots
(M_k)_{\left( \pi(\sigma) \right)(k)_{x,y}}
\end{equation*}
But this means that $M_\sigma$ is a function of $\pi(\sigma)$.

The domain
$\cS_k$ of $\pi$ has cardinality $k!$ while the codomain 
$\Pi^{[k]}$ of $\pi$ has cardinality $|\Pi|^k = c^k$. Since $k$
was chosen such that $k! > c^k$ there must be distinct permutations
$\sigma, \tau \in \cS_k$ such that $\pi(\sigma) = \pi(\tau)$, which
by the previous paragraph means that $M_\sigma = M_\tau$.
\end{proof}

The mistake in \cite{Ales} lies in the proof of the first part of \cite[Proposition 3]{Ales}, where $k$
is taken to be the smallest integer such that $\alpha k^\beta < k!$. The problem is that
$k$ was discussed prior to this point, and in fact played an implicit role in the definition of the set
$\mathcal{C}$, the cardinality of which was in turn used to define $\alpha$ and $\beta$. Thus, one is not necessarily free to choose
$k$ at this point without also changing $\alpha$ and $\beta$. The claim that one may choose $k$ with
$\alpha k^\beta < k!$ implicitly assumes $\alpha$ and $\beta$ to be constant, when in reality they are functions of
$k$ and there is no immediate reason to suppose
that $\alpha k^\beta$ grows more slowly than $k!$.

We discuss briefly the impact upon the correctness of other results in \cite{Ales}. 
The second part of \cite[Proposition 3]{Ales} (which establishes the very important result that finitely generated semigroups of tropical matrices
have polynomial growth) is correct, even though the proof ostensibly employs the same argument as the first part; the erroneous section
of the argument is not required in this part, and the values of $\alpha$ and $\beta$ (and hence also of $\delta$ and $\gamma$)
here are independent of $k$ so that the growth bound obtained really is polynomial in $k$. \cite[Proposition 4]{Ales} is claimed to be proved by ``a slight generalisation'' of the (erroneous) proof of \cite[Proposition 3]{Ales}; we believe a variation on the above proof technique can be used to establish this result, but we do not do this here as it is (not being concerned with bipotent semirings) rather outside the scope of the present paper. The statement of \cite[Proposition 5]{Ales} is true: the main proof given relies on \cite[Proposition 4]{Ales}
and is therefore incomplete, but the alternative proof via Gromov's polynomial growth theorem, outlined in \cite[Remark 3]{Ales}, is valid. 

\section{Strong Permutability}

In this section we turn our attention to the stronger version of permutability.
We shall need the following result, which is trivial where the semiring $S$ has a zero element but requires slightly more work when it does not.
First, recall that for a product of $k$ matrices $M_1\cdots M_k$ and a permutation $\sigma \in \cS_k$, we write $M_\sigma = M_{\sigma(1)}\cdots M_{\sigma(k)}$.
\begin{prop}\label{dimensionchange}
Let $S$ be a commutative bipotent semiring. If $M_n(S)$ is strongly permutable then $M_m(S)$ is strongly permutable for all $m < n$. If $UT_n(S)$ is strongly permutable then $UT_m(S)$ is strongly permutable for all $m < n$.
\end{prop}
\begin{proof}
Consider first the case of full matrix semigroups.
Suppose false for a contradiction; then there is an $m < n$ such that for every $k \in \N$ there exist $m \times m$ matrices $M_1, \dots, M_k$ such that $M_\sigma \neq M_e$
for any non-trivial permutation $\sigma$.
Fix $k$ and let $M_1, \dots, M_k$ be as given. Let $z$ be the smallest (with respect to the order on the semiring) entry of any matrix $M_i$.
For each $i$ let $N_i$ be the $n \times n$ matrix obtained by taking $M_i$ and adjoining $n-m$ rows at the bottom and $n-m$ columns at the right in which every entry is $z$.
\par Now consider the $x,y$ entry of a product $N_{i_1}\cdots N_{i_k}$ for $x,y \leq m$. As $S$ is bipotent this entry is equal to the maximum (with respect to
the order in the semiring) across sequences
$x = x_0, x_1, \dots, x_k = y$ of the term:
$$\prod_{j=1}^{k}{(N_{i_j})}_{x_{j-1},x_j}.$$
If in such a sequence we have $x_j > m$ for some $1 \leq j < k$, then $(N_{i_j})_{x_{j-1},m} \geq z = (N_{i_j})_{x_{j-1},x_j}$ and $(N_{i_{j+1}})_{m,x_{j+1}} \geq z = (N_{i_{j+1}})_{x_j,x_{j+1}}$ by definition, so we may replace $x_j$ by $m$ in the sequence without reducing the resulting term. Thus, we may assume the above
maximum is attained for a sequence with $x_j \leq m$ for all $j$, and it follows that the top-left $m \times m$ submatrix of the product is the product of the corresponding submatrices in the factors, in other words, the corresponding product of the $M_i$s.
In particular, for any permutation $\sigma$ the top-left $m \times m$ submatrix of $N_\sigma$ is exactly $M_\sigma$. Thus, 
$N_\sigma \neq N_e$ for any
non-trivial permutation $\sigma$, which since $k$ was chosen arbitrarily contradicts the assumption that $M_n(S)$ is permutable.
\par For the upper triangular case, there exists a surjective homomorphism from $UT_n(S)$ to $UT_m(S)$ for $m < n$ by only considering the first $m$ rows and columns. Hence if $UT_n(S)$ is permutable then $UT_m(S)$ is permutable for all $m < n$.
\end{proof}

Our next objective is to show that matrix semigroups over a (not necessarily commutative) bipotent semiring with elements of infinite multiplicative order (or more generally, unbounded multiplicative order) are not, in general, permutable. A key tool is a result of Okni\'{n}ski \cite[Chapter 19, Lemma 22]{Okninski}, stating that a finitely generated inverse semigroup with infinitely many idempotents cannot be permutable. In particular this means
that the \textit{bicyclic monoid} is not permutable. This will combine with a representation
of the bicyclic monoid by tropical matrices, due to Izhakian and Margolis \cite{Zur2}, to yield non-permutability results for tropical matrix monoids, and then with our classification
of the monogenic bipotent semirings (Lemma~\ref{Ninside}) to obtain non-permutability results for matrix monoids over semirings with elements of infinite order. Some elementary model theory extends these results to semirings with unbounded order.

\begin{thm} \label{tripermute}
$M_n(\Nmax), \ M_n(\Nmin), UT_n(\Nmax)$ and $UT_n(\Nmin)$ are not strongly permutable for $n \geq 2$.
\end{thm}
\begin{proof}
Let $\cB = \langle p,q \ | \ pq = 1 \rangle$ be the bicyclic monoid. Recall that every element of $\cB$ can be written as $q^ip^j$ for some $i,j \in \N \cup \{0\}$.
By \cite{Zur2} there is a semigroup embedding of $\cB$ into $UT_2(\T)$ given by
\[ \rho : \cB \rightarrow UT_2(\T), \quad q^ip^j \mapsto
\begin{pmatrix}
i - j & i+j \\
-\infty & j-i
\end{pmatrix}.
\]
Since the bicyclic monoid is not permutable \cite[Chapter 19, Lemma 22]{Okninski} and subsemigroups of permutable semigroups are permutable, we deduce that
$UT_2(\T)$ is not permutable. Indeed further, for every $k$ there are upper triangular matrices $M_1, \dots, M_k$ whose
diagonal and above-diagonal entries are integers, with the property that $M_\sigma \neq M_e$ for every non-trivial permutation $\sigma \in \S_k$.

If we fix an integer $\lambda$ strictly less then every integer appearing in these matrices, then the tropically scaled matrices
$(-\lambda) M_1, \dots, (-\lambda) M_k$ clearly also have this property. Replacing the $-\infty$ entry of these matrices with the zero element of $\Nmax^0$ yields a sequences of matrices to show that $UT_2(\Nmax)$ is not strongly permutable.
Similarly, tropically scaling $M_1, \dots, M_k$ by the negative of an integer greater than every entry yields a sequence of matrices
for each $k$ showing that $UT_2(\Nmin)$ is not strongly permutable.

It remains to establish the claims for full matrix semigroups. (Note that, since the semirings here lack zero elements, we do not have a natural embedding of each upper triangular matrix semigroup into the corresponding full matrix semigroup which would allow us to immediately deduce
the remaining claims.)

\par Let $k > 1$ and $M_1, \dots, M_k$ be as above. Choose a very large $\mu \in \N$, and let $N_1, \dots, N_k \in M_n(\Nmax)$ be obtained from $M_1, \dots, M_k$ by scaling tropically by $\mu$, and replacing the $-\infty$ below the diagonal with $1$. Now consider the product $N_\sigma$ for some $\sigma \in \S_k$, and in particular the computation of the $(x,y)$ entry for some $(x,y) \neq (2,1)$. A simple calculation shows that, provided $\mu$ was chosen large enough, the terms which do not feature the (2,1) entry of any $N_i$ will all exceed those which do, from which it follows that $(N_\sigma)_{x,y} = k \mu + (M_\sigma)_{x,y}$. Thus, we conclude that
$N_\sigma \neq N_e$. Since $k$ and $\sigma$ were arbitrary, this means that $M_n(\Nmax)$ is not strongly permutable.

\par Finally, tropically scaling the matrices $N_1, \dots, N_k$ by a sufficiently negative integer gives a sequence to show that $M_n(\Nmin)$ is not strong permutable
\end{proof}

\begin{lem} \label{unitriperm}
$U_n(\Nmax)$ is strongly permutable if and only if $n \leq 2$.
\end{lem}
\begin{proof}
Recall that $U_1(\Nmax)$ is trivial while $U_2(\Nmax)$ is isomorphic to the (commutative) additive semigroup of the semiring, so both are 
strongly permutable. There exists a surjective morphism from $U_n(\Nmax)$ to $U_3(\Nmax)$ for all $n \geq 3$ by mapping to each matrix to its top-left corner 3 by 3 submatrix, 
so it suffices to show that $U_3(\Nmax)$ is not strongly permutable.

\par So, we define the sequence of matrices $B_1, B_2, \dots, B_m$ by
\[ 
B_i = 
\begin{pmatrix}
0 & i & m \\
-\infty & 0  & m + 1 -i \\
-\infty &  -\infty & 0
\end{pmatrix}
\]
(Note that technically speaking $-\infty, 0 \notin \Nmax$; the ``$-\infty$'' and ``$0$'' featured here are technically the zero and identity elements adjoined in
$(\Nmax)^{01}$ which is used in the definition of the unitriangular matrix semigroup $U_3(\Nmax)$, but because this is essentially the same as the
subsemiring $\Nmax \cup \lbrace 0, -\infty \rbrace$ of $\T$ it is clearer to denote them in this way.)
A simple inductive argument shows that for each $k$,
\[ \prod_{i=1}^k B_i = 
\begin{pmatrix}
0 & k & m \\
-\infty & 0 & m \\
-\infty & -\infty  & 0
\end{pmatrix} \]
Now, suppose $\sigma \in \cS_m$ is such that $B_\sigma := \prod_{i=1}^m B_{\sigma(i)} = \prod_{i=1}^m B_i$.
By the definition of matrix multiplication, for any $j < k$ we must have
\[ m \ = \ (B_\sigma)_{1,3} \ \geq \ (B_{\sigma(j)})_{1,2} + (B_{\sigma(k)})_{2,3} \ = \ \sigma(j) + m+1 -\sigma(k) \]
and hence $\sigma(j) < \sigma(k)$. Since $\sigma$ is a permutation, this can only happen if $\sigma$ is the identity permutation. Further, as $m$ was arbitrary no non-trivial permutations preserve this product for any $m \in \N$, so $U_3(\Nmax)$ is not strongly permutable. 
\end{proof}

\begin{lem} \label{minunitriperm}
$U_n(\Nmin)$ is strongly permutable if and only if $n \leq 3$.
\end{lem}
\begin{proof}
Much as in the previous proof, $U_1(\Nmin)$ is the trivial monoid while $U_2(\Nmin)$ is isomorphic to the (commutative) additive semigroup of the semiring, so both are clearly
strongly permutable, and there is a surjective morphism from $U_n(\Nmin)$ to $U_3(\Nmin)$ for all $n \geq 3$, so it suffices to show that $U_3(\Nmin)$ is not strongly permutable.
\par To this end we define the sequence of matrices $C_1, \dots, C_m$ given by
\[ 
C_i = 
\begin{pmatrix}
0 & i-m-1 & -m - 2 \\
-\infty & 0  & -i \\
-\infty & -\infty  & 0
\end{pmatrix}
\]
Once again, the $-\infty$ and $0$ here are formally speaking the zero and identity elements in $(\Nmin)^{01}$. The product of the first $k$ such matrices is inductively seen to be
\[ \prod_{i=1}^k C_i = 
\begin{pmatrix}
0 & k-m-1 & -m-2 \\
-\infty & 0 & -1 \\
-\infty & -\infty  & 0
\end{pmatrix} \]
Now, if $\sigma \in \cS_m$ is such that $C_\sigma := \prod_{i=1}^m C_{\sigma(i)} = \prod_{i=1}^m C_i$ then for any $j < k$,
\[ (C_\sigma)_{1,3} \ = \ -m-2 \ \geq \ (C_{\sigma(j)})_{1,2} + (C_{\sigma(k)})_{2,3} \ = \ \sigma(j) - m - 1 - \sigma(k) \]
so that $\sigma(j) < \sigma(k)$. Since $\sigma$ is a permutation, this can only happen if $\sigma$ is the identity permutation. Further, as $m$ was arbitrary no non-trivial permutations preserve this product for any $m \in \N$, so $U_3(\Nmin)$ is not strongly permutable. 
\end{proof}

\begin{lem} \label{inforder}
Let $S$ be a (not necessarily commutative) bipotent semiring. If $S$ has an element of infinite multiplicative order, then $M_n(S)$ and  $UT_n(S)$ are not strongly permutable for $n \geq 2$ and $U_{n}(S)$ is not strongly permutable if and only if $n \geq 3$.
\end{lem}
\begin{proof}
Suppose $a \in S$ has infinite order. Then by Lemma \ref{Ninside} we have that subsemiring generated by $a$ is isomorphic $\Nmax$ or $\Nmin$. Hence, $M_n(S)$ contains an embedded copy either of $M_n(\Nmax)$ or of $M_n(\Nmin)$; since neither of these are permutable for $n \geq 2$ by Theorem~\ref{tripermute}, $M_n(S)$ is not permutable for $n \geq 2$. Similarly, $UT_n(S)$ is not permutable for $n \geq 2$ using Theorem \ref{tripermute} and $U_{n}(S)$ is not permutable if and only if $n \geq 3$ using Lemma \ref{unitriperm} and Lemma \ref{minunitriperm}.
\end{proof}

A bipotent semiring (even a commutative one) may have elements of unbounded finite order, without having an element of infinite order. For example, we shall see below that the truncated tropical semiring $\T_{[0,1]}$ is such a semiring. Some basic model theory allows us to extend the above result to this case; we direct the reader unfamiliar with model theoretic techniques to for example \cite{Kirby}.

\begin{thm} \label{unbounded}
Let $S$ be a (not necessarily commutative) bipotent semiring with elements of unbounded multiplicative order (that is, such that for all $k \in \N$ there exists an $x \in S$ such that $x$ has multiplicative order greater than $k$). Then the semigroups $M_n(S)$ and $UT_n(S)$ are not strongly permutable for $n \geq 2$.  The semigroup $U_{n}(S)$ is not strongly permutable if and only if $n \geq 3$.
\end{thm}
\begin{proof}
Consider the set of first-order sentences in the language of semirings:
$$L \ = \ \lbrace x^m \neq x^n \mid m, n \in \mathbb{N}, m \neq n \rbrace$$
where $x$ is a variable and $x^m$ is shorthand for the product of $m$ copies of $x$.
Since $S$ has elements of
unbounded order, $L$ is finitely satisfiable (every finite subset of $L$ holds for some $x \in S$) which means that $L$ is a $1$-type of $S$.

\par By realisability of types (see for example \cite[Lemma 23.6]{Kirby}) there exists an elementary extension of $S$ (a
structure containing $S$ and satisfying exactly the same first-order theory) in which $L$ is satisfiable, that is, in which there is an
element $x$ satisfying all of the sentences in $L$. Let $T$ be such a structure and $x \in T$ such an element.
The axioms for a bipotent semiring are clearly all expressible as first-order sentences, so the structure $T$ is itself a
bipotent semiring. Moreover, since $x$ satisfies all sentences in $L$, $x$ is an element of infinite order, and so by
Lemma~\ref{inforder} we deduce that $M_n(T)$ is not permutable.

\par Now suppose for a contradiction that $M_n(S)$ was permutable. This means there exists an $m$ such that
\[\forall X_1,\dots,X_m \in M_2(S), \bigvee_{\sigma \in \cS_m \setminus \{1_m\}} X_1\cdots X_m = X_{\sigma(1)}\cdots X_{\sigma(m)}.\]
Since matrix multiplication is first-order definable in the language of semirings, this can clearly be re-expressed as a first-order sentence over $S$,
featuring $m n^2$ universally quantified scalar variables corresponding to the entries of the $m$ matrices.  But $T$ is elementary equivalent to $S$, so this
sentence also holds in $T$, which contradicts the fact that $M_n(T)$ is not permutable.

Near-identical arguments show that $UT_n(S)$ is not permutable for $n \geq 2$ and that $U_{n}(S)$ is not permutable for $n \geq 3$.
Finally, recall that $U_1(S)$ is trivial while $U_2(S)$ is isomorphic to the additive semigroup of $S$, which is always commutative and hence strongly permutable.
\end{proof}

Recall that $M_1(S) = UT_1(S)$ is isomorphic to the multiplicative semigroup of the semiring $S$. This may be permutable (for example when the semiring is commutative) or non-permutable (for example when $S$ is a non-commutative free monoid with a bipotent addition given by the shortlex total ordering). 
\begin{cor}
Let $S$ be a commutative bipotent semiring with unbounded order. Then $M_n(S)$ (and $UT_n(S)$) are strongly permutable if and only if $n = 1$.
\end{cor}

Recall that a \textit{semifield} is a commutative semiring, possibly without zero, where the non-zero elements form a group with multiplication. In the case of semifields, we can now give an explicit description of when the matrix semigroups are permutable.
\begin{cor} \label{perm}
Let $S$ be a bipotent semifield. Then $M_n(S)$ and $UT_n(S)$ are permutable for $n \geq 2$ (and
$U_{n}(S)$ is permutable for $n \geq 3$) if and only if $S$ is the 2-element boolean semifield.
\end{cor}
\begin{proof}
Since $S$ is a bipotent semiring we have that every element has infinite order or period 1 by Lemma \ref{order}. However $S$ is a semifield, so the non-zero elements form a group with multiplication so the only possible elements of period 1 are the identity and the zero if there is one. Thus, non-identity, non-zero elements are of infinite order. Therefore if $S$ is not the 2-element boolean semifield, it must have an element of infinite order and thus by Theorem \ref{unbounded} (or Lemma~\ref{inforder}), $M_n(S)$ and $UT_n(S)$ are not permutable for $n \geq 2$ and $U_{n}(S)$ is not permutable for $n \geq 3$. If $\B$ is the 2-element boolean semifield then $M_n(\B), UT_n(\B)$, and $U_n(\B)$ are finite and hence permutable for all $n \in \N$.
\end{proof}

\begin{thm} \label{kerperm}
Suppose $S$ is a (not necessarily commutative or bipotent) semiring with the following property: for every finite
subset $X$, there exists a homomorphism to a finite semiring of order bounded by a function in the size of $X$ such that each element of $X$ occupies its own singleton kernel class. Then $M_n(S)$ is permutable for all $n \in \N$.
\end{thm}
\begin{proof}
Let $k$ be such that for every subset $X$ of $S$ with $|X| = n^2$, there is a homomorphism from $S$ to a finite semiring of size at most $k$
such that each element of $X$ occupies its own kernel class. Let $m = k^{n^2} +1$, and suppose
\[ \Sigma =
A_1A_2 \cdots A_m = 
\begin{pmatrix}
x_{1,1} & \dots & x_{1,n} \\
\vdots & \ddots  & \vdots \\
x_{n,1} & \dots & x_{n,n}
\end{pmatrix}
\]
By assumption we may choose a semiring homomorphism $\phi$ mapping $S$ into a semiring $F$ of cardinality at most $k$, such that each $x_{i,j}$ occupies its own kernel class. From this semiring homomorphism, we define a semigroup homomorphism $\psi$ mapping $M_n(S)$ into $M_n(F)$ where
\[ {(\psi(A))}_{i,j} = \phi(A_{i,j}) \text{ for all } i,j.\]
Notice that, since the entries of $\Sigma$ each occupy their own $\phi$-kernel class, $\Sigma$ occupies its own $\psi$-kernel class. 
Since $F$ has cardinality at most $k$, $M_n(F)$ has cardinality at most $k^{n^2} < m$, so there must exist distinct $i$ and $j$ with $\psi(A_i) = \psi(A_j)$. Let $\sigma \in \mathcal{S}_m$ be the transposition swapping $i$ and $j$. Then clearly
$$\psi(A_{\sigma(1)} \dots A_{\sigma(m)}) = \psi(A_{\sigma(1)}) \dots \psi(A_{\sigma(m)}) = \psi(A_1) \dots \psi(A_m) = \psi(\Sigma),$$
which since $\Sigma$ occupies its own $\psi$-kernel class means that
$$A_{\sigma(1)} \dots A_{\sigma(m)} \ = \ \Sigma \ = \ A_1 \dots A_m,$$
as required to show that $M_n(S)$ is permutable.
\end{proof}
Recall that we say a binary relation $\cong$ on a semiring is a \emph{congruence} if $\cong$ is an equivalence relation and if $a \cong b$ and $c \cong d$ together imply that $ac \cong bd$ and $a+c \cong b+d$.
\begin{cor}\label{chains}
Let $S$ be a chain semiring (that is, a totally ordered set with operations maximum and minimum). Then $M_n(S)$ is permutable for all $n \in \N$.
\end{cor}
\begin{proof}
Let $X = \lbrace x_1, \dots, x_k \rbrace$ be a finite subset of $S$. Define a binary relation $\equiv$ on $S$ by $a \equiv b$ if and only if $a$ and $b$ either (i) are equal or (ii) are not in $X$ and lie above exactly the same elements of $X$. Recalling that $S$ is totally ordered, it is easy to see that $\equiv$ is an equivalence relation with at most $2|X| + 1$ classes (being the singleton sets containing elements of $X$, and the open order intervals above, below and between elements of $X$), in which each element of $X$ occupies its own equivalence class. Further, it can be readily seen that $\equiv$ is a congruence. Hence, by the usual first isomorphism theorem, the natural morphism
$S \to S / \equiv$ satisfies the conditions of Theorem~\ref{kerperm}.
\end{proof}
\section{Truncated Tropical Semirings}

In this section we shall illustrate some of the ``wilder'' behaviour which is possible in commutative bipotent semirings, by studying truncated
tropical semirings. To avoid confusion with classical operations, which we shall also need, we use the symbols $\oplus$ and $\otimes$ for the 
denote the addition (maximum) and multiplication (truncated addition) operations in a truncated tropical semiring. The symbol $+$ and juxtaposition
will be used for standard arithmetic addition and multiplication of real numbers, respectively.
We begin by observing that there are a number of isomorphisms between semirings in this class:
\begin{thm} \label{trunciso}
Let $y > x \geq 0$. Then 
\[\T_{[x,y]} \ \cong \ 
\begin{cases} 
\T_{[0,1]} \ \ &\textrm{ if } x = 0; \\
\T_{[1,2]} \ \ &\textrm{ if } x > 0 \textrm{ and } y \leq 2x; \\
\T_{[1,2.5]} \ \ &\textrm{ if } x > 0 \textrm{ and } 2x < y < 3x; \\
\T_{[1,\frac{y}{x}]} &\textrm{ if } x > 0 \textrm{ and } y \geq 3x. 
\end{cases}
\]
The semirings $\T_{[0,1]}, \T_{[1,2]}, \T_{[1,2.5]}$ and $\T_{[1,y]}$ for $y \geq 3$ are pairwise non-isomorphic.
\end{thm}
\begin{proof}
If $x=0$, we define the map $\phi: \T_{[0,y]} \to \T_{[0,1]}$ by
\[ \phi(-\infty) = -\infty \text{ and } \phi(z) = \frac{z}{y} \text{ for } z \in [0,y].\]
As classical multiplication distributes over classical addition, and the fact that $y>0$ implies that $\phi$ is order preserving, it can be easily seen that $\phi$ is an isomorphism. 

\par If $x >0$ and $y \leq 2x$, we define the map $\phi: \T_{[x,y]} \to \T_{[1,2]}$ by
\[ \phi(-\infty) = -\infty, \ \phi(0)= 0, \text{ and } \phi(z) = \frac{z-x}{y-x} + 1 \text{ for } z \in [x,y].\]
Now, for $a,b \in [x,y]$, we have that
\[\phi(a) \otimes \phi(b) = \min\left(\frac{a-x}{y-x} + 1 + \frac{b-x}{y-x} + 1,2\right) = 2 = \phi(a \otimes b)\]
as $a,b \geq x$. Moreover, as $y-x>0$, $\phi$ is order preserving. Hence, it can be easily seen that $\phi$ is an isomorphism.

\par If $x > 0$ and $2x < y < 3x$, we define a piecewise linear map $\phi: \T_{[x,y]} \to \T_{[1,2.5]}$ by
\[ \phi(z) =
\begin{cases}
\frac{z-2x}{2(y-2x)} +2 &\text{if } 2x \leq z \leq y\\
\frac{z-(y-x)}{2(3x-y)} + 1.5 &\text{if } y-x < z < 2x \\
\frac{z-x}{2(y-2x)} +1 &\text{if } x \leq z \leq y-x \\
0 &\text{if } z = 0\\
-\infty &\text{if } z = -\infty
\end{cases}\]
Now, for $a \in [y-x,y]$ and $b \in [x,y]$, we have that
\[\phi(a) \otimes \phi(b) = 2.5 = \phi(y) = \phi(a \otimes b)\]
as $\phi(a) \geq 1.5$ and $\phi(b) \geq 1$. Finally, if $a,b \in [x,y-x]$ then
\begin{align*}
\phi(a) \otimes \phi(b) &= \min\left(\frac{a-x}{2(y-2x)} +1 + \frac{b-x}{2(y-2x)} +1,2.5\right)  \\
&= \min\left(\frac{(a+b)-2x}{2(y-2x)}+2,\frac{y-2x}{2(y-2x)} +2\right) \\
&= \frac{\min(a+b,y)-2x}{2(y-2x)}+2 \\
&= \phi(a \otimes b)
\end{align*}
as $a \otimes b \geq 2x$. Moreover, as $y-2x>0$ and $3x -y > 0$ this implies that $\phi$ is order preserving, and hence it can be easily seen that $\phi$ is an isomorphism. 
\par If $x > 0$ and $y>3$ then we define a map $\phi$ from $\T_{[x,y]}$ to $\T_{[1,\frac{y}{x}]}$ by
\[ \phi(-\infty) = -\infty, \ \phi(0)= 0, \text{ and } \phi(z) = \frac{z}{x} \text{ for } z \in [x,y].\]
As classical multiplication distributes over classical addition, and that $x >0$ implies that $\phi$ is order preserving, it can be easily seen that $\phi$ is an isomorphism.

For non-isomorphism we have to show that $\T_{[0,1]},\T_{[1,2]},\T_{[1,2.5]}$ and $\T_{[1,y]}$ for $y \geq 3$ are pairwise non-isomorphic. We can see that $\T_{[0,1]}$ is not isomorphic to any of the others, as it is the only one with unbounded multiplicative order. Similarly, $\T_{[1,y]}$ has no elements of multiplicative order 3 if and only if $y \leq 2$ so $\T_{[1,2]}$ is not isomorphic to the others. For $\T_{[1,2.5]}$, note that $\T_{[1,y]}$ has no elements of multiplicative order 4 if and only if $y \leq 3$, so $\T_{[1,2.5]}$ can not be isomorphic to any of the others apart from perhaps $\T_{[1,3]}$.
\par For a contradiction, suppose that $\T_{[1,2.5]}$ is isomorphic to $\T_{[1,3]}$ and let $\phi: \T_{[1,2.5]} \rightarrow \T_{[1,3]}$ be an isomorphism. 
As $\phi$ is order-preserving, we have that $\phi(1) = 1$ and $\phi(2.5) = 3$. Similarly, as $\phi$ preserves the semiring multiplication, we can conclude that 
\[\phi(2) = \phi(1) \otimes \phi(1) = 2 \text{ and } \phi(1.5) \cdot 1 = \phi(1.5) \otimes \phi(1) =\phi(2.5) = 3\]
and hence $\phi(1.5) \geq 2 = \phi(2)$
contradicting that $\phi$ is order-preserving. Hence, $\T_{[1,3]}$ and $\T_{[1,2.5]}$ are not isomorphic.

Finally, suppose $z > y \geq 3$ and suppose for a contradiction that there is an isomorphism $\phi: \T_{[1,y]} \rightarrow \T_{[1,z]}$. From the fact that
$\phi$ is a morphism and the definition of multiplication in the two semirings, we have $\phi(a+b) = \phi(a)+\phi(b)$ for all $a,b$ with $a+b \leq y$, and $\phi(1) = 1$. Hence, 
$\phi(2) = \phi(1+1) = \phi(1) + \phi(1) = 2$, and for $1 \leq x \leq y-1$,
\[\phi(x) = \phi(x + 1 - 1) = \phi(x + 1) - \phi(1)  = \phi\left(\frac{x+1}{2}\right) + \phi\left(\frac{x+1}{2}\right) -1\]
A simple inductive argument using this fact shows that $\phi(1 + 2^{-n}) = 1+ 2^{-n}$ for all $n \in \N \cup \lbrace 0 \rbrace$. Indeed, the base case is the fact that $\phi(2) = 2$, while if
the claim holds for some $n$ then taking $x = 1+2^{-n}$ we have $\frac{x+1}{2} = 1+2^{-(n+1)}$. Hence by the above $\phi(1+2^{-n} )= 2 \phi(1+2^{-(n+1)})-1$, so
$\phi(1+2^{-(n+1)}) = \frac{1}{2} (\phi(1+2^{-n})+1) = \frac{1}{2}(1+2^{-n} + 1) = 1+2^{-(n+1)}$ and the claim holds for $n+1$.

Note that for any $a,b$ with $a+b \leq 1$ if $\phi(1+a) = 1+a$ and $\phi(1+b) = 1+b$ then $\phi(1+a+b) = \phi(1+a) + \phi(1+b) - \phi(1) = 1+a+b$. By another
simple induction, we deduce that $\phi$ fixes all finite sums of negative powers of $2$ (in other words, all dyadic rationals) in the interval $[1,2]$. Since the
dyadic rationals are dense in the order, it follows that $\phi$ fixes everything in the interval $[1,2]$.

Finally, since $\phi$ preserves the multiplication in $\T_{[1,y]}$ and $y < z$, it preserves all finite sums which sum to $y$ or less. Since every element in
$[1,y]$ is a finite sum of values in $[1,2]$, it follows that $\phi$ is the identity function on $[1,y]$. Since it is surjective, this means that $y=z$.
\end{proof}

Next we observe that, as a consequence of our earlier results, there are examples of such semirings for which matrix semigroups are not permutable in any rank greater than $1$:
\begin{cor}\label{unboundedtruncated}
The semigroup $M_n(\T_{[0,1]})$ is permutable if and only if $n = 1$.
\end{cor}
\begin{proof}
The semigroup $M_1(\T_{[0,1]})$ is commutative and therefore strongly permutable. For $n > 1$, it is easy to
see that $\T_{[0,1]}$ has elements of unbounded multiplicative order (indeed, for any $j \in \mathbb{N}$ the element
$1/j$ has order $j$), so $M_n(\T_{[0,1]})$ is not strongly permutable by Theorem~\ref{unbounded}.
\end{proof}

By Theorem~\ref{trunciso}, we can now always take truncated tropical semirings to be either of the form $\T_{[0,1]}$ or $\T_{[1,z]}$. Corollary~\ref{unboundedtruncated} gives a full description of when the matrix semigroups $M_n(\T_{[0,1]})$ are permutable, so we now focus on matrix semigroups of form $M_n(\T_{[1,z]})$ for some $z > 1$.

\begin{thm} \label{12perm}
$M_n(\T_{[1,2]})$ is strongly permutable for all $n \in \N$.
\end{thm}
\begin{proof}
We shall show that $\T_{[1,2]}$ satisfies the hypothesis of Theorem~\ref{kerperm}.
Let $X = \lbrace x_1, \dots, x_k \rbrace$ be a finite subset of $\T_{[1,2]}$ and $X' = X \cup \{0,-\infty\}$. 
Define a binary relation $\equiv$ on $\T_{[1,2]}$ by $a \equiv b$ if and only if $a$ and $b$ either (i) are equal or (ii) are not in $X'$ and lie above exactly the same elements of $X'$. It is easy to see that $\equiv$ is an equivalence relation with at most $2|X| + 3$ classes, in which each element of $X$ occupies its own equivalence class.

\par We must now show that $\equiv$ is a congruence. As $\T_{[1,2]}$ is commutative, we only have to show that $\equiv$ is a left congruence.
Let $x \equiv y$. Clearly, if $a = 0$ or $-\infty$, we have that $a \otimes x \equiv a \otimes y$ and $a \oplus x \equiv a \oplus y$. Moreover, if $x = y$, we have that
$a \otimes x \equiv a \otimes y$ and $a \oplus x \equiv a \oplus y$. Hence, as $0,-\infty \in X'$, we can assume that $a,x,y \geq 1$, and thus $a \otimes x = 2  = a \otimes y$.

Further, if $a \geq x,y$ or $a \leq x,y$, then clearly $a \oplus  x \equiv a \oplus y$. On the other hand, if $a$ lies between $x$ and $y$ in the order then $a$, $x$, $y$, $a \oplus x$ and $a \oplus y$ all lie above exactly the same elements of $X'$, giving that $a \oplus x \equiv a \oplus y$. Thus we conclude that $\equiv$ is a congruence.
\par Hence, by the usual first isomorphism theorem, the natural morphism
$\T_{[1,2]} \to \T_{[1,2]} / \equiv$ satisfies the conditions of Theorem~\ref{kerperm}, and $M_n(\T_{[1,2]})$ is strongly permutable for all $n \in \N$.
\end{proof}

The rest of this section treats the remaining truncated tropical semirings, that is, those of the form $\mathbb{T}_{[1,z]}$ with $z > 2$. These will give examples of semirings $S$ such that $M_2(S)$ is strongly permutable, but $M_n(S)$ is not strongly permutable for all $n > 2$. We use the notation $\ceil{z}$ to denote the smallest integer greater than or equal to $z \in \R$.
We shall say that a semigroup $S$ is $k$-permutable if for every $s_1,\dots,s_k \in S$ there exists a non-trivial permutation $\sigma \in \cS_k$ such that $s_{\sigma(1)}s_{\sigma(2)}\cdots s_{\sigma(k)} = s_1s_2\cdots s_k$.

\begin{lem} \label{xperm}
For $z > 2$, let $S$ and $S'$ be subsemigroups of $M_2(\T_{[1,z]})$ given by
\[S = 
\left\{
\begin{pmatrix}
0 &  a \\
-\infty & b 
\end{pmatrix}
: a,b \in \T_{[1,z]}
\right\}
\text { and }
S' = 
\left\{
\begin{pmatrix}
0 &  -\infty \\
a & b 
\end{pmatrix}
: a,b \in \T_{[1,z]}
\right\}.
\]
Then $S$ and $S'$ are both $(2\ceil{z}+5)$-permutable.
\end{lem}
\begin{proof}
Transposing matrices is a semigroup anti-isomorphism between $S$ and $S'$, so it suffices to prove that $S$ is $(2\ceil{z}+5)$-permutable.
\par Let $m = 2\ceil{z} +5$ and let $X_1, \dots, X_m \in S$. If $(X_t)_{2,2}= -\infty$ for any $t>2$ then, as $X_t$ is a right zero of $S$, we have that $X_1X_2\cdots X_m = X_2X_1\cdots X_m$. Thus we may assume $(X_t)_{2,2} \neq -\infty$ for all $t > 2$.
\par If $(X_t)_{1,2},(X_{t+1})_{1,2} = -\infty$ for some $t < m$ then as diagonal matrices commute, we have $X_1\cdots X_tX_{t+1} \cdots X_m = X_1\cdots X_{t+1}X_{t} \cdots X_m$. Therefore, we may assume either $(X_2)_{1,2} \neq -\infty$ or $(X_3)_{1,2} \neq -\infty$. Combined with
the assumption from the previous paragraph, this implies we may assume that $(X_1\cdots X_m)_{1,2} \neq -\infty$.
\par If $(X_t)_{2,2},(X_{t+1})_{2,2} = 0$ for some $t < m$ then, because $2 \times 2$ unitriangular matrices commute, we have $X_1\cdots X_tX_{t+1} \cdots X_m = X_1\cdots X_{t+1}X_{t} \cdots X_m$.
Hence, we may assume that among every pair of every two consecutive matrices (except perhaps the first three) there is a matrix $X_t$ with $(X_t)_{2,2} \geq 1$.
Since $m = 2\ceil{z}+5$ this means we have $(X_1 \cdots X_{m-2})_{1,2} = z$ and $(X_1 \cdots X_{m-2})_{2,2} = z$ or $-\infty$. In both of these cases $X_1 \cdots X_{m-2}$ acts as a left zero for all matrices $M$ with $M_{2,2} \neq -\infty$. But we assumed $(X_t)_{2,2} \neq -\infty$
for $t > 2$, so we have
$$X_1 \cdots X_m = X_1 \cdots X_{m-2} X_{m-1} X_m = X_1 \cdots X_{m-2} X_m X_{m-1}.$$
Thus $S$, and hence also $S'$, is $(2\ceil{z}+5)$-permutable.
\end{proof}

\begin{lem} \label{mideal}
Let $A_0 \in M_2(\T_{[1,z]})$ and $m$ be the minimum finite entry of $A_0$ (or $m = z$ if $A_0$ if all entries are $-\infty$). Let $k \geq 17(16\ceil{z} + 45)$. Then for all $A_1, \dots, A_k \in M_2(\T_{[1,z]})$, either 
\[ (A_0A_1 \cdots A_k)_{i,j} \neq m \text { for all } i,j\]
or there exists a non-trivial $\sigma \in \cS_k$ such that
\[ A_0A_1A_2\cdots A_k = A_0A_{\sigma(1)}A_{\sigma(2)}\cdots A_{\sigma(k)} \]
\end{lem}

\begin{proof}
Consider a product $A_0 A_1 \dots A_k$.
If the product does not contain an $m$ we are done.
Moreover, as $M_2(\T_{[x,z]}\setminus \{0\})$
is an ideal of $M_2(\T_{[1,z]})$ for all $x \in [1,z]$, we may suppose every truncated product $A_0 A_1 \dots A_p$ with $p < k$ has at least one entry equal
to $m$.
\par By the pigeon hole principle there exists a sequence of indices $0 \leq i_0  < \dots < i_n \leq k$ where $n = \ceil{\frac{k}{4}} -1 $ such that each product matrix $A_0A_1\cdots A_{i_j}$ has $m$ in the same position. If this is the (1,2) or the (2,1) position then note that swapping the rows of $A_0$ swaps the rows of the product $A_0A_1 \cdots A_t$ for all $t \leq k$. Therefore if $\sigma$ is a permutation that does not change the product, then $\sigma$ will also preserve the product obtained by swapping $A_0$'s rows. 
Hence, we can assume that the $m$'s are in the (1,1) or (2,2) position. Moreover, by relabelling the rows and columns if necessary, we can assume without loss of generality that $A_0A_1\cdots A_{i_j}$ has $m$ in the (1,1) position for all $0 \leq j \leq n$. 

\par Now consider the matrices defined by
\[ B = A_0\cdots A_{i_0} \text{ and } B_j = A_{i_{j-1}+1}\cdots A_{i_j} \]
for $1 \leq j \leq n$.
Any permutation of this sequence which does not change the product clearly yields a permutation of the original sequence which does not change the product, so it is enough to seek a non-trivial permutation of this sequence.
We define the truncated products $\Sigma_t := BB_1 \cdots B_t$
for $0 \leq t \leq n$.
\par First we consider any matrices $B_i$ whose entries are all either $0$ or $-\infty$. There are only $16$ distinct matrices of this form,
so if more than $16$ of the $B_i$s have this form then the same matrix would appear twice in the sequence, resulting in a non-trivial permutation that preserves the product. Otherwise, since $n = \ceil{\frac{k}{4}}-1 > 17(4\ceil{z}+11)$ the $B_i$s contain a subsequence of $4 \ceil{z}+11$ consecutive matrices not of this form, say $B_p, \dots, B_q$ where $q-p = 4 \ceil{z}+10$.

\par We now define five subsets of $M_2(\T_{[1,z]})$:
\begin{align*}
S &= \left\{
\begin{pmatrix}
0 &  a \\
-\infty & b 
\end{pmatrix} : a,b \in \T_{[1,z]} \right\},
&&S'= \left\{
\begin{pmatrix}
0 & -\infty \\
a & b
\end{pmatrix} : a,b \in \T_{[1,z]} \right\}, \\
T &= \left\{
\begin{pmatrix}
0 & a \\
0 & b 
\end{pmatrix} : a,b \in \T_{[1,z]}\right\},
&&U = \left\{
\begin{pmatrix}
-\infty & a \\
0 & b
\end{pmatrix} : a,b \in \T_{[1,z]} \right\}, \\
V &= \left\{
\begin{pmatrix}
0 & c \\
a & b
\end{pmatrix} : a,b,c \in \T_{[1,z]}\right\}.
\end{align*}
We shall show that the sequence of $B_i$s contains $2\ceil{z}+5$ consecutive matrices either all in $S$ or all in $S'$. From this it will follow by Lemma~\ref{xperm} that there is a permutation of the sequence which preserves the product, as required.

\par Note that $S,S',T \subseteq V$. For $p \leq t \leq q-1$, we have that $(\Sigma_t)_{1,1} = (\Sigma_{t+1})_{1,1} = m$. So, if $(\Sigma_t)_{1,2} = -\infty$, then
in order to ensure $(\Sigma_t B_{t+1})_{1,1} = (\Sigma_{t+1})_{1,1} = m$ we must have $(B_{t+1})_{1,1} = 0$, that is, $B_{t+1} \in V$. Similarly, if $(\Sigma_t)_{1,2} = m$, then $B_{t+1} \in S, T$ or $U$. Otherwise, $(\Sigma_t)_{1,2} > m$ and we have that $B_{t+1} \in S$.

If the matrices $B_p, \dots, B_{p+2\ceil{z}+4}$ are all in $S'$ then we are done. Otherwise, choose $t$ with $p \leq t \leq p+2\ceil{z}+4$ such that
$B_t \notin S'$. Since $(\Sigma_{t-1})_{1,1} = m$ and $\Sigma_t = \Sigma_{t-1} B_t$, this means that $(\Sigma_{t})_{1,2} \neq -\infty$.

Now because $(\Sigma_t)_{11},(\Sigma_t)_{1,2} \geq m$ and $B_{t+1}$ lies in $S$, $T$ or $U$ with (because of the assumption that the entries of $B_{t+1}$
are not all $0$ and $-\infty$)  
either $(B_{t+1})_{1,2} \geq 1$ or $(B_{t+1})_{2,2} \geq 1$, we have that $(\Sigma_{t+1})_{1,2} > m$ and of course by definition we have $(\Sigma_{t+1})_{1,1} \geq m$. Continuing by induction we deduce that $(\Sigma_{i})_{1,2} > m$ for all $i$ with $t+1 \leq i \leq q$. By the remarks in the last paragraph but one, this means
that $B_j \in S$ for all $t+2 \leq j \leq q$, which means the matrices $B_{t+2}, \dots B_{t+1+2\ceil{z}+5}$ are all in $S$, as required.
\end{proof}
\begin{thm} \label{truncperm}
Let $z > 2$. Then $M_2(\T_{[1,z]})$ is strongly permutable.
\end{thm}
\begin{proof}
Consider a product of matrices $A_1\cdots A_n$ for $n \geq 17(4\ceil{z}+1)(16\ceil{z}+45)$
and let $m_t$ be the smallest finite entry in the product of the first $t$ matrices $\Sigma_t = A_1\cdots A_t$. (If all entries of $\Sigma_t$ are $-\infty$, we define $m_t = z$). Note that $m_1 \leq \dots \leq m_n$ as
$M_2(\T_{[x,z]}\setminus \{0\})$
is an ideal of $M_2(\T_{[1,z]})$ for all $x \in [1,z]$.
Further, let $k_1,\dots, k_s$ be all the values such that $m_{k_{j}-1} < m_{k_j}$.
For a contradiction, suppose that there does not exist a non-trivial permutation $\sigma \in \cS_n$ such that $A_1\cdots A_n = A_{\sigma(1)} \cdots A_{\sigma(n)}$.
Then, by Lemma~\ref{mideal}, we have that $s > 1$ and that $k_j - k_{j-1} < 17(16\ceil{z}+45)$ for all $j$ as there is no permutation preserving the product $A_1\cdots A_n$ by assumption.
\par For any $1 \leq j \leq s-4$, consider the five values $m_{k_j} < m_{k_{j+1}} < m_{k_{j+2}} < m_{k_{j+3}} < m_{k_{j+4}}$ and suppose $m_{k_{j+4}} \neq z$ . It is easy to see that each of these five values is either an entry of the matrix $\Sigma_{k_j}$, or else exceeds $m_{k_j}$ by at least $1$. Since there are not five distinct entries in
$\Sigma_{k_j}$
we must therefore have $m_{k_{j+4}} \geq m_{k_j}+1$. Thus as $0 \leq m_{t} \leq z$ for all $t$, we have that $s \leq 4\ceil{z} + 1$.
So as $n > 17(4\ceil{z}+1)(16\ceil{z}+45)$ we have that $k_j - k_{j-1} \geq 17(16\ceil{z}+45)$ for some $2 \leq j \leq s$, giving a contradiction. Therefore, $M_2(\T_{[1,2]})$ is strongly permutable.
\end{proof}

\begin{thm}\label{1zperm}
Let $z > 2$. Then $M_n(\T_{[1,z]})$ is strongly permutable if and only if $n < 3$.
\end{thm}
\begin{proof}
The direct implication is Theorem~\ref{truncperm}. For the converse implication it suffices by Proposition~\ref{dimensionchange} to show that $M_3(\T_{[1,z]})$ is not permutable. We do this by a variation of the method used to prove Lemma~\ref{unitriperm} above.
\par Choose $\varepsilon$ with $0 < \varepsilon < z-2$. For a fixed $m$, we define a sequence of matrices $B_1,B_2,\dots,B_m$ by
\[ 
B_i = 
\begin{pmatrix}
0 & 1 + \frac{i}{m}\varepsilon & 2 + \varepsilon \\
-\infty & 0  & 1+ \varepsilon - \frac{i-1}{m}\varepsilon \\
-\infty & -\infty & 0
\end{pmatrix}
\]
By induction the product of the first $k$ such matrices is given by
\[ \prod_{i=1}^k B_i = \begin{pmatrix}
0 & 1 + \frac{k}{m}\varepsilon & 2 + \varepsilon \\
-\infty & 0 & 1 + \varepsilon \\
-\infty & -\infty & 0
\end{pmatrix} \]
Now, suppose $\sigma \in \cS_m$ is such that $B_\sigma := \prod_{i=1}^m B_{\sigma(i)} = \prod_{i=1}^m B_i$.
By the definition of matrix multiplication, for any $j < k$ we must have
\[ 2+ \varepsilon \ = \ (B_\sigma)_{1,3} \ \geq \ (B_{\sigma(j)})_{1,2} + (B_{\sigma(k)})_{2,3} \ = \  2 + \varepsilon + \frac{\varepsilon}{m} (\sigma(j) - \sigma(k) + 1) \]
and hence $\sigma(j) < \sigma(k)$. Since $\sigma$ is a permutation this means $\sigma$ is the identity permutation. 
 Further, as $m$ was arbitrary $M_3(\T_{[1,z]})$ is not permutable, so together with the previous theorem we get that $M_n(\T_{[1,z]})$ is permutable if and only if $n < 3$. 
\end{proof}

\end{document}